\documentclass{amsart}
\usepackage[utf8]{inputenc}
\usepackage[T1]{fontenc}

\usepackage{amsmath, amssymb, amsfonts, amsthm, mathtools, bm}
\usepackage{cases, amscd, stmaryrd, mathrsfs, yhmath}
\usepackage{upref}

\usepackage{graphicx}
\usepackage{paralist}

\usepackage{hyperref}
\hypersetup{
    colorlinks=true,
    linkcolor=red,      
    citecolor=red,      
    urlcolor=red        
}

\usepackage[misc]{ifsym}

\numberwithin{equation}{section}
\allowdisplaybreaks

\newtheorem{theorem}{Theorem}[section]

\newtheorem{proposition}[theorem]{Proposition}

\newtheorem{remark}[theorem]{Remark}

\newtheorem{propAlph}{Proposition}

\title[Analysis of a Floating Platform Coupled With a Beam]{Analysis of a Model for a Floating Platform Coupled with a Flexible Beam}
\author{Vicente Ocqueteau}
\address{Institut de Math\'ematiques de Bordeaux (IMB), Universit\'e de Bordeaux, 351, Cours de la Lib\'eration - F 33 405, France}
\email{vicente.ocqueteau@math.u-bordeaux.fr}


\date{\today}

\begin{document}
\begin{abstract}
We provide a rigorous mathematical analysis of a coupled system consisting of a floating platform in a fluid of finite depth, clamped to a flexible Euler-Bernoulli beam. The superstructure supports a rigid tip mass at its free end, resulting in a complex multi-physics interaction between potential flow, rigid-body dynamics, and elasticity. We derive the governing equations by coupling the linearised water-wave equations with the dynamics of the floating foundation and the tip-mass payload. The resulting system is formulated as an abstract Cauchy problem in an appropriate Hilbert space. By employing $C_0$-semigroup theory, we establish its well-posedness. Finally, we derive the exact physical energy balance and prove the energy conservation of the system.
\end{abstract}

\maketitle
\section{Introduction}\label{sec:introduction}
Many modern offshore applications share a common structural architecture consisting of a flexible beam clamped to a floating foundation. Such is the case for offshore floating flare towers, crane vessels, meteorological masts and, most notably, wind turbines \cite{Foussekis_2021, Li_FPSO_2026, Wayman_FTW_2006}.

In this work, we propose a complete linear model to describe these offshore configurations by considering a floating foundation coupled with a flexible superstructure. The superstructure consists of a flexible beam supporting a heavy top-side payload, a structural arrangement originally developed for aerospace applications and classically known as the SCOLE (Spacecraft Control Laboratory Experiment) model \cite{Littman_1988_Exact, Littman_1988_Stabilization}. While the structural dynamics of this configuration are well understood when clamped to the ground, its attachment to a floating body fundamentally changes the system's nature: the fluid-structure interaction leads to dynamic boundary conditions, where the motion of the beam's base is governed by the platform, which is in turn driven by the hydrodynamic pressure of the surrounding fluid. To properly capture this coupling and guarantee an exact energy balance, the kinematic and dynamic interface conditions must be precisely formulated.

The mathematical treatment of this problem lies at the intersection of two historically independent lines of research. The study of floating bodies traces back to the linear models by John \cite{John1949OnTM, John1950OnTM} and Ursell \cite{Ursell1949ONTH}. It was not until very recently that well-posedness for such models has been rigorously studied. More precisely, well-posedness for John's problem for a fixed body was proven in \cite{lannes2025posednessfjohnsfloating}, while the case of a freely floating body was established in \cite{lannes2025fjohnmodelcummins} and Ursell's model in \cite{ocqueteau2025initialvalueproblemdescribing}, all of them in the two-dimensional setting, to which we also limit our work. On the other hand, well-posedness of the structural appendage clamped to the ground was established in \cite{WeissZhao2010}. In this work, we bridge these areas by proposing a complete two-dimensional linear model that couples a floating platform with a flexible beam appendage, and providing a rigorous mathematical analysis of the resulting coupled system.

The remainder of this paper is organised as follows. Section \ref{sec:modelling} details the mathematical modelling of the fluid-structure interaction and derives the governing equations. In Section \ref{sec:main_results}, we state the main results of this work, namely the well-posedness and the exact energy balance of the coupled system. Section \ref{sec:operator_form} provides the rigorous functional framework. Finally, Section \ref{sec:proof_well-posendess} presents the proofs of the main results.

\section{Model Description}\label{sec:modelling}
In this section, we derive the governing equations for the fully coupled fluid-structure interaction, leading to the system \eqref{v_q_w_coupled}. We first present the hydrodynamic interaction between the fluid and the platform, and then integrate the mechanical coupling with the flexible superstructure.

\subsection{Coupling the Fluid and the Platform}
Let $\Omega\subset\mathbb{R}^2$ denote a fluid domain consisting of an infinite horizontal strip of finite depth. We decompose its boundary as $\partial\Omega=\mathcal{E}\cup\Gamma_{\rm w}\cup\Gamma_s$, where $\mathcal{E}$ is the free surface, $\Gamma_{\rm w}$ is the wetted surface of the floating object, and $\Gamma_s$ is the seabed. We assume that the seabed $\Gamma_s$ is a Lipschitz curve strictly separated from the free surface. Outside a compact set containing the floating body, we assume the seabed is horizontal. Furthermore, $\Gamma_{\rm w}$ is assumed to be a Lipschitz curve that is strictly separated from the seabed and intersects the free surface non-tangentially. In other words, the interior angles of the fluid domain at the contact points are strictly between $0$ and $\pi$.

We assume the fluid is inviscid and incompressible, so that its velocity field ${\bf U}$ is governed by the Euler equations. Furthermore, the flow is assumed to be irrotational, allowing us to express the velocity field as the gradient of a potential $\varphi:[0,\infty)\times\Omega\to\mathbb{R}$ such that 
\begin{equation}
{\bf U}(t,x,y)=\nabla\varphi(t,x,y)\qquad\qquad\left(t\geqslant 0,\begin{bmatrix}x \\ y\end{bmatrix} \in \Omega\right).
\end{equation}
The configuration of the centre of mass of the floating platform is described by the vector
\begin{equation}
    {\bf q}(t)=\begin{bmatrix}
        s(t) \\ h(t) \\ \theta(t)
    \end{bmatrix}\qquad\qquad(t\geqslant 0),
\end{equation}
where $s(t)$, $h(t)$, and $\theta(t)$ represent the degrees of freedom of the platform, namely surge, heave, and pitch, corresponding to its horizontal, vertical, and rotational motions, respectively. Let $\mathbf{M}_{\rm p} \in \mathbb{R}^{3 \times 3}$ denote the mass-inertia matrix of the platform with respect to the centre of mass ${\bf r}_G$, which is inherently symmetric and positive-definite. Furthermore, let ${\bf K} \in \mathbb{R}^{3 \times 3}$ denote the hydrostatic stiffness matrix, which is symmetric and positive definite for a stable equilibrium, as we shall assume.

We assume small oscillations around the equilibrium state, leading to the linearised water waves equations for the velocity potential $\varphi$  \cite{lannes2013water, zakharov1968stability}. Under this linearisation, the fluid domain $\Omega$ is fixed at its equilibrium position, and the free surface $\mathcal{E}$ reduces to a subset of the plane $\{y=0\}$. Similarly, the equation of motion for ${\bf q}$ can also be linearised around an equilibrium position. The resulting model coupling the velocity potential $\varphi$ and the floating platform ${\bf q}$ is based on the classical linear theory of John \cite{John1949OnTM} (see also \cite{lannes2025fjohnmodelcummins} for a modern mathematical treatment) and is formulated as
\begin{equation}\label{phi_q_coupled} 
\left\{
\begin{alignedat}{2}
    &\Delta\varphi(t,x,y) = 0 && \quad \left(t\geqslant 0, \begin{bmatrix}x\\y\end{bmatrix}\in\Omega\right)\\
    &\frac{\partial^2 \varphi}{\partial t^2}(t,x,0)+g\frac{\partial \varphi}{\partial y}(t,x,0)=0 && \quad \left(t\geqslant 0, x\in\mathcal{E}\right)\\
    &\frac{\partial \varphi}{\partial {\bf n}}(t,x,y)=\dot{\bf q}(t)\cdot\boldsymbol{\nu}(x,y) && \quad \left(t\geqslant 0, \begin{bmatrix}x\\ y\end{bmatrix}\in\Gamma_{\rm w}\right)\\
    &\frac{\partial \varphi}{\partial {\bf n}}(t,x,y)=0 && \quad \left(t\geqslant 0, \begin{bmatrix}x\\ y\end{bmatrix}\in \Gamma_s\right)\\
    &{\bf M}_{\rm p}\ddot{\bf q}(t)+{\bf K}{\bf q}(t)=-\rho_{\rm f}\int_{\Gamma_{\rm w}}\frac{\partial\varphi}{\partial t}\,\boldsymbol{\nu}\,{\rm d}\Gamma + {\bf F}_{\rm ext}(t)\quad && \quad (t\geqslant 0),
\end{alignedat}
\right.
\end{equation}
where $g$ is the gravitational acceleration constant, $\rho_{\rm f}>0$ is the fluid density, and ${\bf n}=[n_1,n_2]^T$ is the unit normal vector pointing outside $\Omega$. The term ${\bf F}_{\rm ext}(t)$ represents any additional external forces and moments applied directly to the floating foundation, such as those arising from mooring lines or control mechanisms. Furthermore, the vector $\boldsymbol{\nu}=[\nu_1,\nu_2,\nu_3]^T$ is defined as
\begin{equation}
    \boldsymbol{\nu}(x,y)=\begin{bmatrix}
        n_1(x,y) \\ 
        n_2(x,y) \\
        \left(\begin{bmatrix} x \\ y \end{bmatrix}-{\bf r}_G\right)^{\perp}\cdot {\bf n}(x,y)
    \end{bmatrix}\qquad\qquad\left(\begin{bmatrix} x \\ y \end{bmatrix} \in \Gamma_{\rm w}\right),
\end{equation}
where ${\bf r}_G=[0,y_G]^T$ is the centre of mass of the platform in its equilibrium position, and ${\bf v}^\perp = [-v_2, v_1]^T$ denotes the orthogonal for any vector ${\bf v}=[v_1,v_2]^T$.

\begin{remark}
    In the linearised model, the free surface $\mathcal{E}$ is a subset of the line $\{y=0\}$. Consequently, the boundary condition on $\mathcal{E}$ in the system \eqref{phi_q_coupled} involves the partial derivative with respect to $y$, which corresponds to the normal derivative in the general nonlinear setting. Moreover, we identify $\mathcal{E}$ with a subset of $\mathbb{R}$, so that we write $x\in\mathcal{E}$ instead of $[x,0]^T\in\mathcal{E}$.
\end{remark}

\begin{remark}\label{remark_phi_zeta}
    In the system \eqref{phi_q_coupled}, the first equation corresponds to the incompressibility condition in the Euler equations. The equation on $\mathcal{E}$ is the result of combining the kinematic condition and the dynamic condition (linearised Bernoulli's law):
    \begin{equation}\label{phi_zeta}
    \left\{
    \begin{alignedat}{2}
        &\frac{\partial \zeta}{\partial t}(t,x)-\frac{\partial \varphi}{\partial y}(t,x,0)=0 && \quad (t\geqslant 0, x\in\mathcal{E})\\
        &\frac{\partial \varphi}{\partial t}(t,x,0)+g\zeta(t,x)=0 &&
    \end{alignedat}
    \right.
    \end{equation}
    where $\zeta$ is the wave elevation with respect to the mean surface. The boundary conditions on $\Gamma_{\rm w}$ and $\Gamma_s$ mean that there is no fluid flowing from $\Omega$ through these solid interfaces (no-penetration conditions). Finally, the equation of motion for ${\bf q}$ corresponds to Newton's second law, where ${\bf K}{\bf q}(t)$ is the Archimedean restoring force applied to the body, the integral term is the force exerted by the hydrodynamic pressure $p_{\rm d} = -\rho_{\rm f} \frac{\partial \varphi}{\partial t}$ on $\Gamma_{\rm w}$, and all other external forces are included in ${\bf F}_{\rm ext}(t)$.
\end{remark}

The next aim is to reduce the boundary value problem for the water waves equations defined in $\Omega$, to a single equation on the free surface $\mathcal{E}$. To this end, we define
\begin{equation}
    v(t,x)=\varphi(t,x,0)\qquad\qquad(t\geqslant 0, x\in\mathcal{E}).
\end{equation}
Next, we decompose the potential into a \textit{Neumann} contribution $\Phi=[\phi_1, \phi_2, \phi_3]^T$ (the elementary Kirchhoff potentials) and a \textit{Dirichlet} contribution $D_\Omega v$ (the wave-induced potential for a fixed body) by writing
\begin{equation}\label{potential_decomp}
\varphi(t,x,y)=\dot{\bf q}(t)\cdot\Phi(x,y)+(D_\Omega v)(t,x,y)\qquad\qquad\left(t\geqslant 0, \begin{bmatrix}x\\y\end{bmatrix}\in\Omega\right),
\end{equation}
where each $\phi_i$ is the solution of
\begin{equation}\label{aux_potentials_pb}
\left\{
\begin{alignedat}{2}
    &\Delta\phi_i(x,y)=0 && \quad \left(\begin{bmatrix}x\\ y\end{bmatrix}\in\Omega\right)\\
    &\frac{\partial \phi_i}{\partial {\bf n}}(x,y)=\nu_i(x,y) && \quad \left(\begin{bmatrix}x\\ y\end{bmatrix}\in\Gamma_{\rm w}\right)\\
    &\frac{\partial \phi_i}{\partial {\bf n}}(x,y)=0 && \quad \left(\begin{bmatrix}x\\ y\end{bmatrix}\in\Gamma_s\right)\\
    &\phi_i(x,0)=0 && \quad \left(x\in\mathcal{E}\right)
\end{alignedat}
\right.
\end{equation}
and $D_\Omega$ is a Dirichlet lifting operator such that, for a suitable function $f:\mathcal{E}\to\mathbb{R}$ (or $\mathbb{C}$), the function $D_\Omega f:\Omega\to\mathbb{R}$ (or $\mathbb{C}$) solves
\begin{equation}\label{Dirichlet_op_pb}
\left\{
\begin{alignedat}{2}
    &\Delta(D_\Omega f)(x,y)=0 && \quad \left(\begin{bmatrix}x\\ y\end{bmatrix}\in\Omega\right)\\
    &\frac{\partial (D_\Omega f)}{\partial {\bf n}}(x,y)=0 && \quad \left(\begin{bmatrix}x\\ y\end{bmatrix}\in\Gamma_{\rm w}\right)\\
    &\frac{\partial (D_\Omega f)}{\partial {\bf n}}(x,y)=0 && \quad \left(\begin{bmatrix}x\\ y\end{bmatrix}\in\Gamma_s\right)\\
    &(D_\Omega f)(x,0)=f(x) && \quad \left(x\in\mathcal{E}\right).
\end{alignedat}
\right.
\end{equation}
Let 
\begin{equation}\label{informal_DtN}
    (\Lambda_\Omega f)(x):=\frac{\partial (D_\Omega f)}{\partial y}(x,0)\qquad\qquad(x\in\mathcal{E}).
\end{equation}

With \eqref{potential_decomp}, \eqref{Dirichlet_op_pb} and \eqref{informal_DtN}, we can recast \eqref{phi_q_coupled} strictly in terms of the free surface and the floating body dynamics:

\begin{equation}\label{v_q_coupled}
    \left\{
    \begin{alignedat}{2}
        &\frac{\partial^2 v}{\partial t^2}(t,x)+g(\Lambda_\Omega v)(t,x)=-g\,\dot{\bf q}(t)\cdot\frac{\partial \Phi}{\partial y}(x,0) && \quad \hspace{-1cm}\left(t\geqslant 0, x \in \mathcal{E}\right)\\
        &({\bf M}_{\rm p}+{\bf M}_{\rm a})\ddot{\bf q}(t)+{\bf K}{\bf q}(t)=-\rho_{\rm f}\int_{\Gamma_{\rm w}}\frac{\partial(D_\Omega v)}{\partial t}\,\boldsymbol{\nu}\,{\rm d}\Gamma +{\bf F}_{\rm ext}(t),
    \end{alignedat}
    \right.
\end{equation}
where
\begin{equation}
    {\bf M}_{\rm a} = \rho_{\rm f}\int_{\Gamma_{\rm w}} \boldsymbol{\nu} \Phi^T \,{\rm d}\Gamma
\end{equation}
is the so-called \textit{added mass} matrix. With this formulation, the primary unknowns for the floating platform model are $v$ and ${\bf q}$. Once the system is solved, the full potential $\varphi$ can be retrieved via \eqref{potential_decomp}. Moreover, the elevation $\zeta$ can also be recovered by \eqref{phi_zeta}.

\subsection{Coupling the Floating Platform and the Flexible Beam}
We consider a flexible beam of length $L>0$, clamped to the floating platform and featuring a point-mass at its tip. To maintain consistency with the fluid domain, where the mean free surface is located at $y=0$, we position the horizontal axis such that $x=0$ corresponds to the clamping point, and we denote its vertical coordinate by $y_0$. Consequently, the point-mass is located at $y_L = y_0 + L$. 

\begin{remark}
    By setting the clamping point at $x=0$, we implicitly assume that the beam is aligned with the horizontal centre of mass of the floating platform ($x_G = 0$). This is a crucial modelling assumption. If the beam were clamped at an eccentric position ($x_c \neq 0$), the vertical internal forces, such as the axial compression due to the weight and vertical inertia of the tower and point-mass, would exert a non-zero torque on the platform. This would strongly couple the rigid-body heave and pitch motions. Aligning the tower with the centre of buoyancy/mass ensures static stability and mathematically decouples these inertial effects.
\end{remark}

Let $w(t,y)$ denote the transverse (horizontal) deflection of the beam. In this work, we adopt the \textit{absolute formulation}, meaning $w(t,y)$ represents the total absolute displacement of the beam in the inertial frame, i.e., the sum of elastic deformation and the rigid-body motion dragged by the platform. Since we describe the beam in the inertial reference frame, no fictitious inertial forces appear in the governing equations.

The mechanical properties of the beam are defined by its mass density per unit length $\rho(y)$ and its flexural rigidity $EI(y)$, where $E$ is the Young's modulus and $I(y)$ is the area moment of inertia of the cross-section. To ensure a physically meaningful and well-posed formulation, both $\rho$ and $EI$ are assumed to be strictly positive functions of class $C^4([y_0,y_L])$. At its free end $y=y_L$, the beam supports a rigid point-mass (representing, for example, a flare tower payload, a crane vessel load, or a wind turbine rotor) with mass $m > 0$ and moment of inertia $J > 0$.

The internal forces and moments exerted by the platform on the beam at the clamping point are represented by the vector
\begin{equation}
    {\bf F}_{\rm int}[w](t)=\begin{bmatrix}
        \frac{\partial}{\partial y}\left(EI(y)\frac{\partial^2 w}{\partial y^2}(t,y)\right)\\[1.2em]
        0\\[1.2em]
        -EI(y)\frac{\partial^2 w}{\partial y^2}(t,y)+(y-y_G)\frac{\partial}{\partial y}\left(EI(y)\frac{\partial^2w}{\partial y^2}(t,y)\right)
    \end{bmatrix}_{y=y_0}.
\end{equation}
Physically, each component of ${\bf F}_{\rm int}[w]$ captures a specific interaction mechanism \cite{Morgul1990}. The first component is the shear force acting as a horizontal load (coupled with the surge motion). The second component vanishes, reflecting the assumption of an inextensible beam in the linearised model. The third component is the total torque, comprising the internal bending moment at the clamp plus the torque generated by the shear force acting at the lever arm $(y_0-y_G)$ (coupled with the pitch motion).

By Newton's third law, the corresponding back-action exerted by the beam on the rigid platform is exactly $-{\bf F}_{\rm int}[w](t)$. While the beam's back-action acts as an external force on the isolated platform in \eqref{v_q_coupled}, it constitutes an internal interface interaction for the coupled platform-beam system. By isolating the reaction $-{\bf F}_{\rm int}[w](t)$ from the remaining external forces ${\bf F}_{\rm ext}(t)$, the fully coupled system reads as follows:
\begin{equation}\label{v_q_w_coupled}
\mathmakebox[0.9\textwidth][l]{
\left\{
    \begin{aligned}
        &\frac{\partial^2 v}{\partial t^2}(t,x)+g(\Lambda_\Omega v)(t,x)=-g\,\dot{\bf q}(t)\cdot\frac{\partial \Phi}{\partial y}(x,0) && \hspace{-1.5cm}(t\geqslant 0, x \in \mathcal{E})\\
        &{\bf M}_{\rm T}\ddot{\bf q}(t)+{\bf K}{\bf q}(t)=-\rho_{\rm f}\int_{\Gamma_{\rm w}}\frac{\partial(D_\Omega v)}{\partial t}(t,\cdot,\cdot)\,\boldsymbol{\nu}\,{\rm d}\Gamma-{\bf F}_{\rm int}[w](t) + {\bf F}_{\rm ext}(t)\\
        &\rho(y)\frac{\partial^2 w}{\partial t^2}(t,y)+\frac{\partial ^2}{\partial y^2}\left(EI(y)\frac{\partial^2 w}{\partial y^2}(t,y)\right)= 0 &&\hspace{-1.5cm} (y\in(y_0,y_L))\\
        &w(t,y_0)={\bf r}(y_0)\cdot {\bf q}(t), \quad \frac{\partial w}{\partial y}(t,y_0)={\bf r}'(y_0)\cdot {\bf q}(t) \\
        &m\frac{\partial^2 w}{\partial t^2}(t,y_L)-\frac{\partial}{\partial y}\left(EI(y)\frac{\partial^2 w}{\partial y^2}(t,y)\right)\bigg|_{y=y_L}=F_{\rm tip}(t) \\
        &J\frac{\partial^3 w}{\partial y\partial t^2}(t,y_L)+EI(y_L)\frac{\partial^2 w}{\partial y^2}(t,y_L)=M_{\rm tip}(t)
    \end{aligned}
\right.
}
\end{equation}
where ${\bf M}_{\rm T}={\bf M}_{\rm p}+{\bf M}_{\rm a}$ is the total mass matrix of the platform and $F_{\rm tip}(t)$ and $M_{\rm tip}(t)$ represent external forces and torques applied to the tip mass, accounting for environmental loads (e.g., aerodynamic thrust on a wind turbine rotor) or passive and active control strategies.

The third equation in \eqref{v_q_w_coupled} is the classical Euler-Bernoulli model for the beam's transverse vibrations. The boundary conditions in the fourth line dictate the kinematic matching at the rigid clamp $y_0$ via the vector ${\bf r}(y)=[1,0,y-y_G]^T$. Finally, the fifth and sixth equations balance the linear and rotational inertial forces of the tip point-mass with the restoring shear force and bending moment, respectively \cite{Littman_1988_Exact}.

\begin{remark}\label{rem:notation}
To simplify the notation, partial derivatives with respect to time and spatial variables are denoted by subscripts when convenient.
\end{remark}

\section{Statement of the Main Results}\label{sec:main_results}

Before stating the main theorems regarding the well-posedness and energy conservation of the coupled system \eqref{v_q_w_coupled}, we first introduce the rigorous functional framework for the fluid-structure interface. 

Throughout this paper, $H^s(\mathcal{E})$ denotes the standard Sobolev space of order $s\in\mathbb{R}$ defined on the free surface $\mathcal{E}$. The expression \eqref{informal_DtN} for $\Lambda_\Omega$ provides a formal definition valid for smooth functions. However, to establish a rigorous functional framework, we need to extend this operator to appropriate Sobolev spaces. Thanks to the geometric conditions imposed on the wetted surface $\Gamma_{\rm w}$ in Section \ref{sec:modelling}, we have the following key properties.

\begin{propAlph}[{\cite[Corollary~2]{lannes2025posednessfjohnsfloating}}]\label{prop:DtN_properties}
   For every $f \in H^1(\mathcal{E})$, the system \eqref{Dirichlet_op_pb} admits a unique solution $D_\Omega f \in \dot{H}^1(\Omega) = \{ u \in L_{\rm loc}^1(\Omega) \mid \nabla u \in L^2(\Omega)^2 \}$. Moreover, the Dirichlet-to-Neumann map associated with \eqref{Dirichlet_op_pb}, defined by
   \begin{equation}\label{formal_DtN}
      \Lambda_\Omega f : x \mapsto \frac{\partial (D_\Omega f)}{\partial y}(x,0) \qquad (x \in \mathcal{E}),
   \end{equation}
   is well-defined, with $\Lambda_\Omega f \in L^2(\mathcal{E})$. Furthermore, $\Lambda_\Omega$ defines a self-adjoint, positive operator in $L^2(\mathcal{E})$ with domain $H^1(\mathcal{E})$.
\end{propAlph}
Henceforth, $\mathbb{I}$ denotes the identity operator on the relevant Hilbert space, which will be clear from the context.

\begin{propAlph}[{\cite[Lemma 5.10]{ocqueteau2025initialvalueproblemdescribing}}]\label{lem:fractional_domain}
    Let $\Lambda_\Omega$ be the Dirichlet-to-Neumann operator defined above. The domain of its square root coincides with the fractional Sobolev space on the interface
    \begin{equation}
       D((\mathbb{I} + \Lambda_\Omega)^{\frac12}) = H^{\frac12}(\mathcal{E}),
    \end{equation}
    with equivalent norms.
\end{propAlph}

\begin{remark}
    Although the operator $\Lambda_\Omega$ in \cite{ocqueteau2025initialvalueproblemdescribing} was associated with a specific fluid domain geometry, the proof given there relies entirely on abstract operator interpolation theory and the strict positivity of $\mathbb{I} + \Lambda_\Omega$. Consequently, the result is independent of the spatial configuration and applies directly to the more general geometry considered in this work.
\end{remark}

With the properties of the Dirichlet-to-Neumann operator established, we can now state our fundamental well-posedness result for the fully coupled system.

\begin{theorem}\label{well-posedness_v_q_w}
Let $v_0\in H^1(\mathcal{E})$, $v_1\in H^{\frac12}(\mathcal{E})$, ${\bf q}_0, {\bf q}_1\in\mathbb{C}^3$, $w_0\in H^4(y_0,y_L)$, and $w_1\in H^2(y_0,y_L)$. Consider the system \eqref{v_q_w_coupled} with initial conditions
    \begin{equation}\label{initial_cond_mapping}
    \left\{
    \begin{aligned}
        &v(0,\cdot)=v_0, \ \ v_t(0,\cdot)=v_1, \\ &{\bf q}(0)={\bf q}_0, \ \ \dot{\bf q}(0)={\bf q}_1, \\ 
        &w(0,\cdot)=w_0, \ \ w_t(0,\cdot)=w_1.
    \end{aligned}
    \right.
    \end{equation}
    Assume that these initial data satisfy the geometric compatibility conditions at the clamping point:
    \begin{equation}\label{strong_geometric_cond}
    \left\{
        \begin{aligned}
            &w_0(y_0) = \mathbf{r}(y_0)\cdot \mathbf{q}_0,  \quad w_0'(y_0) = \mathbf{r}'(y_0)\cdot \mathbf{q}_0, \\
            &w_1(y_0) = \mathbf{r}(y_0)\cdot \mathbf{q}_1,  \quad w_1'(y_0) = \mathbf{r}'(y_0)\cdot \mathbf{q}_1.
        \end{aligned}
    \right.
    \end{equation}
    Then, for any ${\bf F}_{\rm ext} \in C^1([0,\infty); \mathbb{C}^3)$ and $F_{\rm tip}, M_{\rm tip} \in C^1([0,\infty); \mathbb{C})$, the system admits a unique strong solution $(v,{\bf q}, w)$ satisfying
    \begin{equation}\label{strong_solution_spaces}
    \hspace{-0.9cm}
    \left\{
        \begin{aligned}
            &v\in C([0,\infty); H^1(\mathcal{E}))\cap C^1([0,\infty); H^{\frac12}(\mathcal{E}))\cap C^2([0,\infty); L^2(\mathcal{E})),\\
            &{\bf q}\in C^2([0,\infty); \mathbb{C}^3),\\
            &\mathrlap{w \in C([0,\infty); H^4(y_0,y_L))\cap C^1([0,\infty); H^2(y_0,y_L))\cap C^2([0,\infty); L^2(y_0,y_L)).}
        \end{aligned}
    \right.
    \end{equation}

    Moreover, if the initial data satisfy the lower regularity $v_0\in H^{\frac12}(\mathcal{E})$, $v_1\in L^2(\mathcal{E})$, $w_0\in H^2(y_0, y_L)$, and $w_1\in L^2(y_0,y_L)$, and only the initial position is constrained by
    \begin{equation}\label{weak_geometric_cond}
        w_0(y_0)=\mathbf{r}(y_0)\cdot \mathbf{q}_0, \quad w_0'(y_0)=\mathbf{r}'(y_0)\cdot \mathbf{q}_0,
    \end{equation}
    then for ${\bf F}_{\rm ext} \in L^1_{\rm loc}([0,\infty); \mathbb{C}^3)$ and $F_{\rm tip}, M_{\rm tip} \in L^1_{\rm loc}([0,\infty); \mathbb{C})$, the system \eqref{v_q_w_coupled} with initial conditions \eqref{initial_cond_mapping} admits a unique weak solution $(v,{\bf q}, w)$ satisfying
    \begin{equation}\label{weak_solution_spaces}
    \left\{
        \begin{aligned}
            &v\in C([0,\infty); H^{\frac12}(\mathcal{E}))\cap C^1([0,\infty); L^2(\mathcal{E})),\\
            &{\bf q}\in C^1([0,\infty); \mathbb{C}^3),\\
            &w \in C([0,\infty); H^2(y_0,y_L))\cap C^1([0,\infty); L^2(y_0,y_L)).
        \end{aligned}
    \right.
    \end{equation}
\end{theorem}

Relying on the regularity of the solutions provided by the previous result, we can establish the exact physical energy balance of the coupled system, which reads as follows.

\begin{theorem}\label{thm:energy}
    Let $(v, {\bf q}, w)$ be a weak solution of the coupled system \eqref{v_q_w_coupled}. Then the total physical energy of the system is
    \begin{equation}\label{physical_energy}
    \begin{aligned}
        E(t) &= \frac{\rho_{\rm f}}{2g} \|v_t\|_{L^2(\mathcal{E})}^2 + \frac12 \rho_{\rm f} \|\Lambda_\Omega^{\frac12} v\|_{L^2(\mathcal{E})}^2 \\
        &\quad + \frac12 \langle {\bf M}_{\rm T} \dot{\bf q}, \dot{\bf q} \rangle_{\mathbb{C}^3} + \frac12 \langle {\bf K} {\bf q}, {\bf q} \rangle_{\mathbb{C}^3} \\
        &\quad + \frac12 \|\sqrt{\rho} w_t\|_{L^2(y_0,y_L)}^2 + \frac12 \|\sqrt{EI} w_{yy}\|_{L^2(y_0,y_L)}^2 \\
        &\quad + \frac12 m|w_t(t, y_L)|^2 + \frac12 J |w_{yt}(t, y_L)|^2,
    \end{aligned}
    \end{equation}
    and it satisfies the balance equation
    \begin{equation}\label{energy_balance_formula}
    E(t) = E(0) + \int_0^t \mathrm{Re} \langle {\bf F}_{\rm ext}(\tau), \dot{\bf q}(\tau) \rangle_{\mathbb{C}^3} \, {\rm d}\tau + \int_0^t \mathrm{Re} \big[ F_{\rm tip}(\tau) \bar{w}_t(\tau, y_L) + M_{\rm tip}(\tau) \bar{w}_{yt}(\tau, y_L) \big] \, {\rm d}\tau.
\end{equation}

    In particular, if the system is unforced (${\bf F}_{\rm ext} \equiv {\bf 0}$, $F_{\rm tip} \equiv 0$, and $M_{\rm tip} \equiv 0$), the total physical energy is exactly conserved for all time $t \geqslant 0$.
\end{theorem}

\section{Operator Form of the Governing Equations}\label{sec:operator_form}

Our strategy to prove the well-posedness of \eqref{v_q_w_coupled} is to formulate it as a first-order abstract Cauchy problem $\dot{z}(t) = \mathcal{A}z(t)$ on a suitable Hilbert space $X$. To this end, we first define a strictly positive, self-adjoint operator $A$ on a Hilbert space $H$ that incorporates the free-surface water waves, the restoring forces, the structural coupling between the platform and the beam and the tip point-mass dynamics. We then construct a skew-adjoint generator $\mathcal{A}_0$, and finally treat the remaining hydrodynamic interaction terms as a bounded perturbation $\mathcal{P} \in \mathscr{L}(X)$ such that $\mathcal{A} = \mathcal{A}_0 + \mathcal{P}$.

Let $H$ be the Hilbert space defined by
\begin{equation}
    H = \underbrace{L^2(\mathcal{E}) \vphantom{L^2(y_0,y_L)}}_{\mathclap{\text{Fluid} \vphantom{p}}} \times 
        \underbrace{\mathbb{C}^3 \vphantom{L^2(y_0,y_L)}}_{\mathclap{\text{Platform} \vphantom{p}}} \times 
        \underbrace{L^2(y_0,y_L)\times\mathbb{C}\times\mathbb{C}}_{\mathclap{\text{Beam and point-mass} \vphantom{p}}}
\end{equation}
We endow $H$ with the following weighted inner product:
\begin{equation}\label{inner_product_H}
\begin{aligned}
    \left\langle \begin{bmatrix}v \\ \mathbf{q} \\ w \\ \xi \\ \eta \end{bmatrix}, \begin{bmatrix}\tilde v \\ \tilde{\mathbf{q}} \\ \tilde w \\ \tilde \xi \\ \tilde \eta\end{bmatrix}\right\rangle_H 
    &= \rho_{\rm f} \langle v, \tilde{v} \rangle_{L^2(\mathcal{E})} + g \langle \mathbf{M}_{\rm T} \mathbf{q}, \tilde{\mathbf{q}} \rangle_{\mathbb{C}^3} \\
    &\quad + g \langle \rho w, \tilde{w} \rangle_{L^2(y_0,y_L)} + g m \xi \bar{\tilde{\xi}} + g J \eta \bar{\tilde{\eta}}.
\end{aligned}
\end{equation}

To introduce the rigorous physical coupling, we recall the kinematic influence vector $\mathbf{r}(y) = [1, 0, y - y_G]^T$. We now define the operator $A:D(A) \subset H \to H$. To ensure $A$ is self-adjoint with respect to \eqref{inner_product_H}, its domain must incorporate the absolute geometric compatibility conditions between the platform, the beam, and the tip point-mass:
\begin{equation}
    D(A) = \left\{ \left.\begin{bmatrix}v \\ \mathbf{q} \\ w \\ \xi \\ \eta\end{bmatrix} \in H^1(\mathcal{E}) \times \mathbb{C}^3 \times H^4(y_0,y_L) \times \mathbb{C} \times \mathbb{C} \ \ \right| \ \ \begin{aligned} 
    &\xi = w(y_L)\\ 
    &\eta = w'(y_L) \\  
    &w(y_0)= \mathbf{r}(y_0) \cdot \mathbf{q}\\
    &w'(y_0)= \mathbf{r}'(y_0) \cdot \mathbf{q} \end{aligned} \right\}.
\end{equation}
The action of $A$ is given by
\begin{equation}
     A \begin{bmatrix} v \\ \mathbf{q} \\ w \\ \xi \\ \eta \end{bmatrix} = \begin{bmatrix} 
        g\Lambda_{\Omega}v + v \\ 
        \mathbf{M}_{\rm T}^{-1} \left( {\bf K} \mathbf{q} + {\bf F}_{\rm int}[w] \right)\\ 
        \rho^{-1} (EI w_{yy})_{yy} \\ 
        -m^{-1} (EI w_{yy})_y(y_L) \\ 
        J^{-1}(EI w_{yy})(y_L)
     \end{bmatrix}.
\end{equation}

\begin{proposition}\label{prop_A_self_adjoint}
    The operator $A: D(A) \subset H \to H$ is self-adjoint and strictly positive. 
\end{proposition}

\begin{proof}
    In view of Proposition \ref{lem:fractional_domain}, the natural space for the trace of the velocity potential on the interface is the fractional Sobolev space $H^{\frac12}(\mathcal{E})$. This motivates the introduction of the auxiliary space $V$, defined as the closed subspace of the product space $H^{\frac12}(\mathcal{E}) \times \mathbb{C}^3 \times H^2(y_0,y_L)\times\mathbb{C}\times\mathbb{C}$ given by
    \begin{equation}\label{V_space}
        V = \left\{\left.\begin{bmatrix}v \\ {\bf q} \\ w \\ \xi \\ \eta\end{bmatrix} \in H^{\frac12}(\mathcal{E}) \times \mathbb{C}^3 \times H^2(y_0,y_L)\times\mathbb{C}\times\mathbb{C} \ \ \right| \ \ \begin{aligned} 
            &\xi = w(y_L) \\ 
            &\eta = w'(y_L)\\
            &w(y_0)= \mathbf{r}(y_0) \cdot {\bf q}\\
            &w'(y_0)= \mathbf{r}'(y_0) \cdot {\bf q}
        \end{aligned} \right\}.
    \end{equation}
    Equipped with the norm inherited from this product, $V$ is a Hilbert space that is continuously and densely embedded in $H$. On this space, we define the symmetric bilinear form $a: V \times V \to \mathbb{C}$ as
    \begin{equation}
    \begin{aligned}
        a(U, \tilde{U}) &= \rho_{\rm f} g \langle \Lambda_{\Omega}^{\frac12} v, \Lambda_{\Omega}^{\frac12} \tilde{v} \rangle_{L^2(\mathcal{E})} + \rho_{\rm f} \langle v, \tilde{v} \rangle_{L^2(\mathcal{E})} \\
        &\quad+ g \langle {\bf K} {\bf q}, \tilde{\bf q} \rangle_{\mathbb{C}^3} + g \langle EI w_{yy}, \tilde{w}_{yy} \rangle_{L^2(y_0,y_L)}.
    \end{aligned}
    \end{equation}
    We will show that $a(\cdot, \cdot)$ is a continuous and strictly coercive form on $V$.

    The continuity of the form follows directly from the Cauchy-Schwarz inequality, as each term is bounded by the corresponding product of norms in $H^{\frac12}(\mathcal{E})$, $\mathbb{C}^3$, and $H^2(y_0, y_L)$. To establish coercivity, we show that $a(U,U)$ controls the full $H^2$-norm of $w$. Integrating from the clamped boundary $y_0$ and using the kinematic conditions, we have:
    \begin{equation}
        w_y(y) = \mathbf{r}'(y_0) \cdot {\bf q} + \int_{y_0}^y w_{yy}(\sigma) \, {\rm d}\sigma, \quad w(y) = \mathbf{r}(y_0) \cdot {\bf q} + \int_{y_0}^y w_y(\sigma) \, {\rm d}\sigma.
    \end{equation}
    Applying the Cauchy-Schwarz inequality, we obtain the following estimates:
    \begin{equation}
    \begin{aligned}
        \|w_y\|_{L^2(y_0,y_L)}^2 &\leqslant C_1 |{\bf q}|^2 + C_2 \|w_{yy}\|_{L^2(y_0,y_L)}^2, \\
        \|w\|_{L^2(y_0,y_L)}^2 &\leqslant C_3 |{\bf q}|^2 + C_4 \|w_{yy}\|_{L^2(y_0,y_L)}^2.
    \end{aligned}
    \end{equation}
    These inequalities ensure that the $H^2$-norm of $w$ is bounded by the state variables already present in $a(U,U)$. Moreover, the continuity of the trace operator from $H^2(y_0, y_L)$ to $\mathbb{C}$ implies 
    \begin{equation}
|\xi|^2 + |\eta|^2 \leqslant C_5 \|w\|_{H^2(y_0,y_L)}^2.
\end{equation}
Since $D((\mathbb{I}+\Lambda_\Omega)^{\frac12}) = H^{\frac12}(\mathcal{E})$ (Proposition \ref{lem:fractional_domain}) and ${\bf K}$ is positive definite, there exists $\alpha > 0$ such that $a(U,U) \geqslant \alpha \|U\|_V^2$.
    
   These properties establish that $a(\cdot, \cdot)$ is a continuous and strictly coercive symmetric bilinear form on the Hilbert space $V$. Since $V$ is continuously embedded in $H$, the mapping $\tilde{U} \mapsto \langle F, \tilde{U} \rangle_H$ defines a bounded linear functional on $V$ for any given $F = [f_v, f_{\bf q}, f_w, f_\xi, f_\eta]^T \in H$. Thus, the Lax-Milgram theorem guarantees that there exists a unique $U \in V$ such that
    \begin{equation} \label{eq:var_equality}
        a(U, \tilde{U}) = \langle F, \tilde{U} \rangle_H \qquad\qquad (\tilde{U} \in V).
    \end{equation}
    
    To show that this unique solution $U$ actually belongs to the operator domain $D(A)$ and that $AU = F$, we recover the spatial regularity by restricting \eqref{eq:var_equality} to specific test functions. Choosing $\tilde{U} = [0, {\bf 0}, \tilde{w}, 0, 0]^T$ with $\tilde{w} \in C_c^\infty(y_0, y_L)$ yields   
    \begin{equation}
        g \langle EI w_{yy}, \tilde{w}_{yy} \rangle_{L^2(y_0,y_L)} = \langle f_w, \tilde{w} \rangle_{L^2(y_0,y_L)}. 
    \end{equation} 
    This implies $g(EI w_{yy})_{yy} = f_w \in L^2(y_0, y_L)$ in the sense of distributions, which guarantees $w \in H^4(y_0, y_L)$. Similarly, choosing $\tilde{U} = [\tilde{v}, {\bf 0}, 0, 0, 0]^T$ gives 
    \begin{equation}
        \rho_{\rm f} \langle g\Lambda_{\Omega} v + v, \tilde{v} \rangle_{L^2(\mathcal{E})} = \rho_{\rm f} \langle f_v, \tilde{v} \rangle_{L^2(\mathcal{E})}, 
    \end{equation}
    forcing $g\Lambda_{\Omega} v + v = f_v \in L^2(\mathcal{E})$. Because $g\Lambda_{\Omega} + \mathbb{I}$ is a strictly positive, self-adjoint operator on $L^2(\mathcal{E})$ with domain $H^1(\mathcal{E})$, it is a bijective mapping from $H^1(\mathcal{E})$ onto $L^2(\mathcal{E})$. This strictly guarantees that $v \in H^1(\mathcal{E})$.

    With the $H^4$-regularity of $w$ established, integration by parts over $(y_0, y_L)$ is now justified for a general $\tilde{U} \in V$. Using the kinematic constraints dictated by $V$, we substitute the integrated beam term back into \eqref{eq:var_equality}:
    \begin{equation}
    \begin{aligned}
        &\langle g {\bf K} {\bf q} - g(EI w_{yy})_y(y_0)\mathbf{r}(y_0) + g(EI w_{yy})(y_0)\mathbf{r}'(y_0), \tilde{\bf q} \rangle_{\mathbb{C}^3} \\
        &\quad + g(EI w_{yy})_y(y_L) \bar{\tilde{\xi}} - g(EI w_{yy})(y_L) \bar{\tilde{\eta}} = \langle f_{\bf q}, \tilde{\bf q} \rangle_{\mathbb{C}^3} + f_\xi \bar{\tilde{\xi}} + f_\eta \bar{\tilde{\eta}}.
    \end{aligned}
    \end{equation}
    Since $\tilde{\bf q}$, $\tilde{\xi}$, and $\tilde{\eta}$ can be chosen arbitrarily and independently in $\mathbb{C}^3 \times \mathbb{C} \times \mathbb{C}$, the boundary terms must match the right-hand side component-wise. This recovers exactly the dynamic boundary conditions defined in $D(A)$, proving that $U \in D(A)$ and $AU = F$. 
    
    Moreover, the relation $\langle AU, \tilde{U} \rangle_H = a(U, \tilde{U})$ holds for all $U \in D(A)$ and $\tilde{U} \in V$. The symmetry of the bilinear form $a(\cdot, \cdot)$ therefore directly implies that $\langle AU, \tilde{U} \rangle_H = \langle U, A\tilde{U} \rangle_H$ for all $U, \tilde{U} \in D(A)$, meaning $A$ is a symmetric operator. Since we have also shown that for every $F \in H$ there is a unique $U \in D(A)$ satisfying $AU = F$, the operator $A$ is surjective. Thus, $A$ is self-adjoint. Finally, the strict positivity of $A$ on its domain follows immediately from the coercivity bound of $a(\cdot, \cdot)$.
\end{proof}

Since $A$ is a strictly positive self-adjoint operator, its fractional power $A^{\frac12}$ is well-defined. Furthermore, its domain $D(A^{\frac12})$ coincides exactly with the space $V$ defined in \eqref{V_space}. We can now rigorously define the total state space $X$ as
\begin{equation}
    X = D(A^{\frac12}) \times H,
\end{equation}
endowed with the inner product
\begin{equation}  
    \left\langle \begin{bmatrix}U_1\\ \tilde U_1\end{bmatrix},\begin{bmatrix}U_2\\ \tilde U_2\end{bmatrix}\right\rangle_X =\left\langle A^{\frac12}U_1,A^{\frac12}U_2 \right\rangle_H + \left\langle \tilde U_1,\tilde U_2\right\rangle_H.
\end{equation}

By the properties of the associated bilinear form, the potential energy part of the inner product reduces directly to $\langle A^{\frac12} U, A^{\frac12} U \rangle_H = a(U,U)$. Thus, the total squared norm in $X$ for an element $z = [U, \tilde{U}]^T$ is explicitly given by:
\begin{equation}\label{norm_X}
\begin{aligned}
    \|z\|_X^2 &= \rho_{\rm f} g \|\Lambda_{\Omega}^{\frac12} v\|_{L^2(\mathcal{E})}^2 + \rho_{\rm f} \|v\|_{L^2(\mathcal{E})}^2 \\
    &\quad + g \langle {\bf K} {\bf q}, {\bf q} \rangle_{\mathbb{C}^3} + g \| \sqrt{EI} w_{yy} \|_{L^2(y_0,y_L)}^2 \\
    &\quad + \rho_{\rm f} \|\tilde{v}\|_{L^2(\mathcal{E})}^2 + g \langle {\bf M}_{\rm T} \tilde{\bf q}, \tilde{\bf q} \rangle_{\mathbb{C}^3} \\
    &\quad + g \|\sqrt{\rho}\tilde{w}\|_{L^2(y_0,y_L)}^2 + g m |\tilde{\xi}|^2 + g J |\tilde{\eta}|^2.
\end{aligned}
\end{equation}

We now define $\mathcal{A}_0$ as
\begin{gather}
    D(\mathcal{A}_0)=D(A)\times D(A^{\frac12})\\
    \mathcal{A}_0=\begin{bmatrix}
        0 & \mathbb{I}\\
        -A & 0
    \end{bmatrix}.
\end{gather}
By Proposition \ref{prop_A_self_adjoint}, the operator $\mathcal{A}_0$ is skew-adjoint on $X$ and generates a strongly continuous unitary group $(\mathbb{T}_t)_{t \in \mathbb{R}}$ on $X$ (see \cite[Proposition~3.7.6]{Obs_Book}).

Our goal now is to define a suitable coupling operator $\mathcal{P} \in \mathscr{L}(X)$ and set $\mathcal{A} = \mathcal{A}_0 + \mathcal{P}$, ensuring that the abstract first-order evolution equation
\begin{equation}
    \dot{z}(t) = \mathcal{A}z(t), \qquad (t \geqslant 0)
\end{equation}
provides a consistent representation of the original coupled system \eqref{v_q_w_coupled}, with the state vector defined as:
\begin{equation}
    z(t) = \begin{bmatrix} 
        \rule{0pt}{1.6em} [v(t,\cdot), \mathbf{q}(t), w(t,\cdot), w(t,y_L), w_y(t,y_L)]^T \\ 
        [v_t(t,\cdot), \dot{\mathbf{q}}(t), w_t(t,\cdot), w_t(t,y_L), w_{yt}(t,y_L)]^T \rule[-1.2em]{0pt}{0pt}
    \end{bmatrix}.
\end{equation}

To account for the physical interaction terms and remove the artificial identity term $\mathbb{I}$ introduced to ensure strict positivity, we define the perturbation operator $\mathcal{P}: X \to X$. Let $z = [U, \tilde U]^T \in X$, where $U = [v, \mathbf{q}, w, \xi, \eta]^T \in D(A^{\frac12})$ and $\tilde U = [\tilde{v}, \tilde{\mathbf{q}}, \tilde{w}, \tilde{\xi}, \tilde{\eta}]^T \in H$. We define $\mathcal{P}$ as the block operator:
\begin{equation}
    \mathcal{P} \begin{bmatrix} U \\ \tilde U \end{bmatrix} = \begin{bmatrix} \mathbf{0} \\ P \begin{bmatrix} U \\ \tilde U \end{bmatrix} \end{bmatrix},
\end{equation}
where $P: X \to H$ captures the remaining interaction forces. Additionally, we define the purely hydrodynamic acceleration of the platform as:
\begin{equation}
    \mathbf{a}_{\text{hyd}}[\tilde{v}] =-\mathbf{M}_{\rm T}^{-1}\left(  \rho_{\rm f} \int_{\Gamma_{\rm w}} (D_\Omega \tilde{v})\boldsymbol{\nu} \, {\rm d}\Gamma\right).
\end{equation}

To ensure sufficient regularity for the hydrodynamic interactions, we must now restrict our attention to platform geometries that do not induce strong singularities. The action of $P$ is defined entirely by the hydrodynamic interactions, as the structural couplings are fully captured by the domain of $\mathcal{A}_0$. The first component of $P$ accounts for the platform's velocity coupling and cancels the regularizing identity term $\mathbb{I}$ introduced in the definition of $A$.
\begin{equation}\label{P_formula}
    P \begin{bmatrix} U \\ \tilde U \end{bmatrix} = \begin{bmatrix} 
        v - g \,\tilde{\mathbf{q}} \cdot \Phi_y(\cdot, 0) \\
        \mathbf{a}_{\text{hyd}}[\tilde{v}] \\
        0 \\
        0 \\
        0
    \end{bmatrix}.
\end{equation}

\begin{proposition}\label{prop_P_bounded}
    The operator $\mathcal{P}$ is bounded on $X$, i.e., $\mathcal{P} \in \mathscr{L}(X)$.
\end{proposition}
\begin{proof}
    Let $z = [U, \tilde{U}]^T \in X$. We must show that $\|\mathcal{P}z\|_X \leqslant C \|z\|_X$ for some constant $C>0$. By the definition of the norm in $X$, $\|\mathcal{P}z\|_X = \|P z\|_H$. We estimate the non-zero components of $P z$ given in \eqref{P_formula}. 

Regarding the fluid component, we must evaluate the regularity of the Kirchhoff potentials $\phi_i$ near the contact points to ensure their normal derivatives admit an $L^2$ trace on the free surface. As established in \cite[Proposition 3.6]{lannes2025fjohnmodelcummins}, since the interior angles are strictly between $0$ and $\pi$, the normal derivatives of the potentials are locally square-integrable on the free surface (see also \cite{Dauge1988}). Combined with the global $L^2$-regularity of the gradient on Lipschitz boundaries away from the corners \cite[Theorem 2.1]{Brown1994}, we conclude that $\Phi_y(\cdot, 0) \in (L^2(\mathcal{E}))^3$.
    
    Therefore, the first component of $\mathcal{P}z$ is bounded by
    \begin{equation}
        \|v - g \,\tilde{\mathbf{q}} \cdot \Phi_y(\cdot, 0)\|_{L^2(\mathcal{E})} \leqslant \|v\|_{L^2(\mathcal{E})} + g |\tilde{\mathbf{q}}|\,\|\Phi_y(\cdot,0)\|_{L^2(\mathcal{E})^3} \leqslant C_1 \|z\|_X,
    \end{equation}
    where we used that the $L^2$-norm of $v$ and the vector norm of $\tilde{\mathbf{q}}$ are continuously controlled by the $X$-norm of $z$.

    Second, for the platform's hydrodynamic acceleration $\mathbf{a}_{\text{hyd}}[\tilde{v}]$, we assess the integral $\int_{\Gamma_{\rm w}} (D_\Omega \tilde{v}) \nu_i \, {\rm d}\Gamma$. Applying Green's second identity between the harmonic functions $D_\Omega \tilde{v}$ and the Kirchhoff potentials $\phi_i$, and exploiting the homogeneous boundary conditions of $\phi_i$ on $\mathcal{E}$ and the normal derivative of $D_\Omega \tilde{v}$ on $\Gamma_{\rm w}$, we obtain:
    \begin{equation}
        \int_{\Gamma_{\rm w}} (D_\Omega \tilde{v}) \nu_i \, {\rm d}\Gamma = -\int_{\mathcal{E}} \tilde{v}(x) \frac{\partial \phi_i}{\partial y}(x, 0) \, {\rm d}x.
    \end{equation}
    Since $(\phi_i)_y(\cdot, 0) \in L^2(\mathcal{E})$, the Cauchy-Schwarz inequality yields:
    \begin{equation}
        |\mathbf{a}_{\text{hyd}}[\tilde{v}]| \leqslant C_2 \|\tilde{v}\|_{L^2(\mathcal{E})} \|\Phi_y(\cdot, 0)\|_{L^2(\mathcal{E})^3} \leqslant C_3 \|\tilde{v}\|_{L^2(\mathcal{E})} \leqslant C_4 \|z\|_X.
    \end{equation}
    Combining these estimates, we conclude that $\|P z\|_H \leqslant C \|z\|_X$, so that $\mathcal{P} \in \mathscr{L}(X)$.
\end{proof}

\section{Proof of the Main Results}\label{sec:proof_well-posendess}
We are now in a position to prove our first main result, Theorem \ref{well-posedness_v_q_w}.

\begin{proof}[Proof of Theorem \ref{well-posedness_v_q_w}]
    Let us define the state variable 
    \begin{equation}
        z(t) = \begin{bmatrix} 
            \rule{0pt}{1.6em} [v(t,\cdot), \mathbf{q}(t), w(t,\cdot), w(t,y_L), w_y(t,y_L)]^T \\ 
            [v_t(t,\cdot), \dot{\mathbf{q}}(t), w_t(t,\cdot), w_t(t,y_L), w_{yt}(t,y_L)]^T \rule[-1.2em]{0pt}{0pt}
        \end{bmatrix}.
    \end{equation}
    By combining Proposition \ref{prop_A_self_adjoint} and Proposition \ref{prop_P_bounded}, standard bounded perturbation theory implies that the full operator $\mathcal{A} = \mathcal{A}_0 + \mathcal{P}$ generates a strongly continuous group $(\mathbb{T}_t^{\mathcal{P}})_{t\in\mathbb{R}}$ on the energy space $X$ (see \cite[Theorem~2.11.2]{Obs_Book}). Next, let us define the source term
    \begin{equation}
\mathbf{f}(t)= \big[0, {\bf 0}, 0, 0, 0, \ 0, \mathbf{M}_{\rm T}^{-1}{\bf F}_{\rm ext}(t), 0, m^{-1}F_{\rm tip}(t), J^{-1}M_{\rm tip}(t)\big]^T \qquad (t \geqslant 0)
    \end{equation}
    and consider the abstract Cauchy problem
    \begin{equation}\label{abstract_first_order}
        \left\{
        \begin{aligned}
            &\dot z(t) = \mathcal{A}z(t) + \mathbf{f}(t) \qquad (t \geqslant 0)\\
            &z(0) = z_0
        \end{aligned}
        \right..
    \end{equation}

    We first prove the existence of a strong (classical) solution. The higher regularity and the full set of geometric compatibility conditions \eqref{strong_geometric_cond} imposed on the initial data $[v_0, \mathbf{q}_0, w_0]^T$ and $[v_1, \mathbf{q}_1, w_1]^T$ in the theorem statement correspond precisely to the requirement that the initial state $z_0 = z(0)$ belongs to the operator domain $D(\mathcal{A})$. Since ${\bf F}_{\rm ext}\in C^1([0,\infty);\mathbb{C}^3)$ and $F_{\rm tip}, M_{\rm tip} \in C^1([0,\infty); \mathbb{C})$, we have $\mathbf{f} \in C^1([0,\infty); X)$. Standard semigroup theory (see, e.g., \cite{Pazy1983,Obs_Book}) ensures that under these conditions, the problem \eqref{abstract_first_order} admits a unique strong solution $z \in C([0,\infty); D(\mathcal{A})) \cap C^1([0,\infty); X)$ given by 
    
     \begin{equation}\label{semigroup_formula}
        z(t) = \mathbb{T}_t^{\mathcal{P}}z_0 + \int_0^t \mathbb{T}_{t-\tau}^{\mathcal{P}}\mathbf{f}(\tau) \, {\rm d}\tau\qquad\qquad(t\geqslant 0).
    \end{equation}
    
    Translating this regularity back to the individual physical variables directly yields the solution spaces specified in \eqref{strong_solution_spaces}.

   Finally, we address the lower regularity case. The relaxed initial conditions and the single geometric constraint \eqref{weak_geometric_cond} correspond to the requirement that $z_0$ belongs merely to the energy space $X$. For ${\bf F}_{\rm ext} \in L_{\rm loc}^1([0,\infty); \mathbb{C}^3)$ and $F_{\rm tip}, M_{\rm tip} \in L^1_{\rm loc}([0,\infty); \mathbb{C})$, we have $\mathbf{f} \in L_{\rm loc}^1([0,\infty); X)$. In this scenario, \eqref{abstract_first_order} possesses a unique weak (mild, i.e., of the form \eqref{semigroup_formula}) solution $z \in C([0,\infty); X)$. Mapping this regularity back to the physical variables yields the weak solution spaces stated in \eqref{weak_solution_spaces}.
\end{proof}

In standard semigroup theory, the energy conservation of a system is typically an automatic consequence of the skew-adjointness of its generator. However, in our framework, the skew-adjoint operator $\mathcal{A}_0$ constructed in Section \ref{sec:operator_form} does not represent the physical energy of the target system \eqref{v_q_w_coupled}. Instead, it corresponds to a modified model that omits specific hydrodynamic interaction terms and artificially includes the identity operator $\mathbb{I}$ on the fluid component to ensure strict positivity. As a result, the full generator $\mathcal{A} = \mathcal{A}_0 + \mathcal{P}$ is not skew-adjoint on $X$, and the abstract norm in $X$ does not capture the exact physical energy. 
    
To obtain a direct physical energy balance via the skew-adjointness of the exact operator, one would need to formulate the fluid problem using the homogeneous Beppo-Levi space $\dot{H}^1(\Omega)$ for the interior potential and its associated semi-normed trace space on the free surface, in order to handle the lack of coercivity over constant modes (see \cite{lannes2025posednessfjohnsfloating,lannes2025fjohnmodelcummins}). We preferred an approach relying on familiar, standard Sobolev spaces. While this deliberate choice requires us to verify the exact physical energy balance by hand, the well-posedness result established in Theorem \ref{well-posedness_v_q_w} rigorously guarantees that all norms and traces appearing in the physical energy $E(t)$ are well-defined and finite for the corresponding initial conditions.

\begin{proof}[Proof of Theorem \ref{thm:energy}]
Assume first that the initial data are in $D(\mathcal{A})$ and ${\bf F}_{\rm ext}\in C^1([0,\infty);\mathbb{C}^3)$, $F_{\rm tip}, M_{\rm tip} \in C^1([0,\infty); \mathbb{C})$ so that $(v, {\bf q}, w)$ is a classical solution satisfying \eqref{strong_solution_spaces}. For simplicity of notation, we omit the time dependence of the terms.
    
    Applying the operation $\frac{\rho_{\rm f}}{g}\mathrm{Re}\langle \cdot, v_t\rangle_{L^2(\mathcal{E})}$ to the fluid equation, we obtain
    \begin{equation}\label{energy_fluid_proof}
        \frac{{\rm d}}{{\rm d}t}\left( \frac{\rho_{\rm f}}{2g} \|v_t\|_{L^2(\mathcal{E})}^2 + \frac12 \rho_{\rm f} \|\Lambda_\Omega^{\frac12}v\|_{L^2(\mathcal{E})}^2\right) = -\rho_{\rm f} \mathrm{Re} \langle \dot{\bf q}\cdot \Phi_y(\cdot,0),v_t\rangle_{L^2(\mathcal{E})}.
    \end{equation}
    
    Next, taking the inner product of the platform equation with $\dot{\bf q}$ in $\mathbb{C}^3$ and extracting the real part yields
    \begin{equation}\label{energy_platform_proof}
    \begin{aligned}
        \frac{{\rm d}}{{\rm d}t}\left(\frac12 \langle {\bf M}_{\rm T}\dot{\bf q},\dot{\bf q}\rangle_{\mathbb{C}^3} + \frac12 \langle {\bf K}{\bf q},{\bf q}\rangle_{\mathbb{C}^3}\right) &= - \rho_{\rm f} \mathrm{Re} \left\langle \int_{\Gamma_{\rm w}} (D_\Omega v_t)\boldsymbol{\nu}\,{\rm d}\Gamma, \dot{\bf q} \right\rangle_{\mathbb{C}^3} \\
        &\quad - \mathrm{Re} \langle {\bf F}_{\rm int}[w], \dot{\bf q}\rangle_{\mathbb{C}^3} + \mathrm{Re} \langle {\bf F}_{\rm ext}, \dot{\bf q}\rangle_{\mathbb{C}^3}.
    \end{aligned}
    \end{equation}
    Applying Green's identity to the hydrodynamic pressure term and exploiting the homogeneous boundary conditions of the Kirchhoff potentials $\phi_i$ on $\mathcal{E}$ and the normal derivative of $D_\Omega v_t$ on $\Gamma_{\rm w}$, we have
    \begin{equation}
      \int_{\Gamma_{\rm w}} (D_\Omega v_t)\nu_i \,{\rm d}\Gamma = -\int_{\mathcal{E}} v_t \frac{\partial \phi_i}{\partial y}(\cdot,0) \,{\rm d}x.  
    \end{equation}
    Thus, the first term on the right-hand side of \eqref{energy_platform_proof} becomes precisely 
    \begin{equation}
     \rho_{\rm f} \mathrm{Re} \langle \dot{\bf q}\cdot \Phi_y(\cdot,0), v_t \rangle_{L^2(\mathcal{E})},   
    \end{equation}
    which perfectly cancels the corresponding term on the right-hand side of \eqref{energy_fluid_proof} when the two equations are added.

    Finally, we take the inner product of the beam equation with $w_t$ in $L^2(y_0,y_L)$. Taking the real part and integrating the bending stiffness term by parts twice yields boundary terms at both $y=y_L$ and $y=y_0$. By substituting the dynamic boundary equations governing the tip point-mass at $y_L$, those terms are absorbed into the time derivative, yielding:
    \begin{equation}\label{energy_beam_proof}
    \begin{aligned}
        &\frac{{\rm d}}{{\rm d}t}\left(\frac12 \|\sqrt{\rho}w_t\|_{L^2(y_0,y_L)}^2 + \frac12 \|\sqrt{EI}w_{yy}\|_{L^2(y_0,y_L)}^2 + \frac12 m|w_t(y_L)|^2 + \frac12 J|w_{yt}(y_L)|^2\right) \\
        &\qquad =  \mathrm{Re} \left[ (EI w_{yy})_y(y_0) \bar{w}_t(y_0) - (EI w_{yy})(y_0) \bar{w}_{yt}(y_0) \right] + \mathrm{Re} \big[ F_{\rm tip} \bar{w}_t(y_L) + M_{\rm tip} \bar{w}_{yt}(y_L) \big].
    \end{aligned}
    \end{equation}
    Differentiating the kinematic clamping conditions at the base with respect to time, we obtain $w_t(y_0) = \mathbf{r}(y_0)\cdot \dot{\bf q}$ and $w_{yt}(y_0) = \mathbf{r}'(y_0)\cdot \dot{\bf q}$. By inserting these relations and expanding ${\bf q} = [s, h, \theta]^T$, the right-hand side of \eqref{energy_beam_proof} matches exactly with $\mathrm{Re} \langle {\bf F}_{\rm int}[w], \dot{\bf q} \rangle_{\mathbb{C}^3}$. This perfectly cancels the remaining internal structural interaction term in \eqref{energy_platform_proof}. 

    Upon summing equations \eqref{energy_fluid_proof}, \eqref{energy_platform_proof}, and \eqref{energy_beam_proof}, all internal coupling terms vanish, leaving the differential balance equation
    \begin{equation}
        \frac{{\rm d}}{{\rm d}t}E(t) = \mathrm{Re} \langle {\bf F}_{\rm ext}(t), \dot{\mathbf{q}}(t) \rangle_{\mathbb{C}^3} + \mathrm{Re} \big[ F_{\rm tip}(t) \bar{w}_t(t, y_L) + M_{\rm tip}(t) \bar{w}_{yt}(t, y_L) \big].
    \end{equation}
    Integrating this relation from $0$ to $t$ yields the integral energy balance identity \eqref{energy_balance_formula} for classical solutions.

Consider now a mild solution $z \in C([0, \infty); X)$ associated with initial data $z_0 \in X$. Since $D(\mathcal{A})$ is dense in $X$, this solution is characterized as the strong limit in $C([0, \infty); X)$ of a sequence of classical solutions. Given that the physical energy functional $E(t)$ is continuously controlled by the $X$-norm topology, the integral energy identity \eqref{energy_balance_formula} extends to all mild solutions by a standard density argument.
\end{proof}
\section*{Acknowledgements}
This work is funded by the European Union (Horizon Europe MSCA project ModConFlex, grant number 101073558).

\end{document}